\documentclass[12pt,a4paper]{amsart}
 \usepackage{amssymb,amsthm}
 \usepackage{amsmath,amssymb,amsfonts,amsthm} 
\usepackage{color}                    
\usepackage{verbatim}   
\usepackage{subfigure}  

\newtheorem{theorem}{Theorem}[section]

\newtheorem{corollary}[theorem]{Corollary}
\newtheorem{lemma}[theorem]{Lemma}

\begin{document}
 \title{SAW$^{*}$-algebras are essentially non-factorizable}
 \author{ Saeed Ghasemi}
 \thanks{The author is indebted to Ilijas Farah for several suggestions and useful conversations. I would also like to thank the referee for suggesting several improvements.}

\address{Department of Mathematics, York University, Toronto, CA}
 \date{}
 \maketitle

\begin{abstract}
In this paper we solve a question of Simon Wassermann, whether the Calkin algebra can be written as a C*-tensor product of two infinite dimensional C*-algebras. More generally we show that there is no surjective *-homomorphism from a $SAW^{*}$-algebra onto C*-tensor product of two infinite dimensional C*-algebras.
\end{abstract}

\section{Introduction}

 It was shown by L. Ge \cite{Ge}, using free entropy,  that if the group von Neumann algebra of $\mathbb{F}_{2}$, $L(\mathbb{F}_{2})$,  is written as the von Neumann tensor product of two von Neumann algebras $M$ and $N$ then  either $M$ or $N$ has to be isomorphic to the algebra of $n\times n$ matrices  $\mathbb{M}_{n}(\mathbb{C})$ for some $n$.
 For two C*-algebras $\mathcal{A}$ and $\mathcal{B}$ the C*-algebra tensor product is not unique and for a C*-norm $\|.\|_{\nu}$ on the algebraic tensor product $\mathcal{A}\odot\mathcal{B}$ the completion is usually denoted by $\mathcal{A}\otimes_{\nu}\mathcal{B}$ (see \cite{Black}). A C*-algebra $\mathcal{A}$ is called \emph{essentially non-factorizable} if it can not be written as $\mathcal{B}\otimes_{\nu} \mathcal{C}$ where both $\mathcal{B}$ and $\mathcal{C}$ are infinite dimensional for any C*-algebra norm $\nu$. In a presentation at the London Mathematical Society
meeting held in Nottingham in 2010, Simon Wassermann demonstrated that the reduced C*-algebra of $\mathbb{F}_{2}$, $C^{*}_{r}(\mathbb{F}_{2})$, is essentially non-factorazable. In fact if $C^{*}_{r}(\mathbb{F}_{2})= \mathcal{B}\otimes_{\nu} \mathcal{C}$, for some C*-norm $\nu$ and infinite dimensional C*-algebra $\mathcal{B}$ then $\mathcal{C}=\mathbb{M}_{n}(\mathbb{C})$ with $n=1$. It was asked by Wassermann whether the Calkin algebra is essentially non-factorizable. We prove that the answer to this question is positive, by showing that all $SAW^{*}$-algebras, of which the Calkin algebra is an example, are
essentially non-factorizable.

It is well known that C*-algebras can be viewed as non-commutative topological spaces and the correspondence $X \leftrightarrow C(X)$ is a contravariant category equivalence between the category of compact Hausdorff spaces and continuous maps and the category of commutative unital C*-algebras and unital *-homomorphisms. Each property of a locally compact Hausdorff space can be reformulated in terms of the function algebra $C_{0}(X)$, so it usually make sense to ask about these properties for noncommutative C*-algebras. $SAW^{*}$-algebras were introduced by G.K. Pedersen \cite{PedSAW} as non-commutative analogues of sub-Stonean spaces (also known as F-spaces) in topology, which are the locally compact Hausdorff spaces in which disjoint $\sigma$-compact open subspaces have disjoint compact closure. Analogously a C*-algebra $\mathcal{A}$ is called an $SAW^{*}$-algebra if for every two orthogonal elements $x$ and $y$ in $\mathcal{A}_{+}$ , there is an element $e$ in $\mathcal{A}_{+}$ such that $ex=x$ and $ey=0$. It is not hard to see that an abelian C*-algebra $C_{0}(X)$ is a $SAW^{*}$-algebra if and only if $X$ is a sub-Stonean space.

In \cite{PedCor} and \cite{PedSAW} some of the properties of sub-Stonean spaces are generalized to $SAW^{*}$-algebras. It is proved  ( cf. \cite{PedSAW}) that the corona algebra of any $\sigma$-unital C*-algebra is a $SAW^{*}$-algebra. In particular for a separable Hilbert space the Calkin algebra is a $SAW^{*}$-algebra. Another class of C*-algebras, called countably degree-1 saturated C*-algebras, are introduced in \cite{FH} using model theoretic notions. It was shown that countably degree-1 saturated C*-algebras contain all coronas of $\sigma$-unital C*-algebra, ultrapowers of C*-algebras and relative commutants of  separable subalgebras of a countably degree-1 saturated C*-algebra, and they are all $SAW^{*}$-algebras. In this paper we will use another property of sub-Stonean spaces to show that $SAW^{*}$-algebras are essentially non-factorizable. This will show that the ultrapowers of C*-algebras  and relative commutants of  separable subalgebras of a countably degree-1 saturated C*-algebra are also essentially non-factorizable. A similar result for ultrapowers of type II$_{1}$-factors with respect to a free ultrafilter is proved in \cite{FGL}.\\
 In this paper we don't require any knowledge about sub-Stonean spaces and it's enough to know that $\beta \mathbb{N}$, the Stone-\v{C}ech compactification of $\mathbb{N}$, is a sub-Stonean space.

\section{$SAW^{*}$-algebras are essentially non-factorizable}

We adopt standard notations from Ramsey theory and write $[\mathbb{N}]^2$ to denote the set of all $(m,n)\in \mathbb{N}^2$ such that $m<n$ and $\Delta^{2} \mathbb{N}$ to denote the diagonal of $\mathbb{N}^{2}$. For spaces $X$ and $Y$ a rectangle is a subset of $X\times Y$ of the form $A\times B$ for $A\subset X$ and $B\subset Y$. We say a map $f$ on $A\times B$ depends only on the first coordinate if $f(x,y)=f(x,z)$ for every $(x,y)$ and $(x,z)$ in $A\times B$.
In [11, lemma 5.1] Van Douwen proved that for any continuous map $f: {\beta\mathbb{N}}^{2}\rightarrow \beta\mathbb{N}$ there is a clopen $U\subset \beta\mathbb{N}$ such that $f\upharpoonright U^{2}$ depends on at most one coordinate and conjectured [11, conjecture 8.4] that there is a disjoint open cover of $\beta\mathbb{N}^{2}$  into such sets.   In [4, theorem 3] I. Farah showed that for a sub-Stonean space $Z$, compact spaces $X$ and $Y$, every continuous map $f: X\times Y\rightarrow Z$ is of a "$very ~simple$" form, which will be clear from theorem \ref{10} (in fact the theorem is proved for a larger class of spaces so called the $\beta\mathbb{N}$- spaces in the range and arbitrary powers of a compact space in the domain. However the theorem remains true if products of arbitrary compact spaces is replaced in the domain of the map). We sketch the proof of this theorem for the convenience of the reader. Before we need the following lemma. \\
\begin{lemma}\label{51}
 Suppose $X, Y$ and $Z$ are arbitrary sets, $\rho: X\times Y\rightarrow Z$ a map, then exactly one of the following holds:
 \begin{enumerate}
  \item $X\times Y$ can be covered by finitely many mutually disjoint rectangles such that $\rho$ depends on at most one coordinate on each of them.
   \item There are sequences ${x_{i}}\in X$, ${y_{i}}\in Y$ such that for all $i$ and all
$j<k$ we have $\rho({x_{i}},{y_{i}})\neq \rho({x_{j}},{y_{k}})$.
\end{enumerate}
Moreover if $X, Y$ and $Z$ are topological spaces and $\rho$ is a continuous map, we can assume that the rectangles in $(1)$ are clopen.
\end{lemma}
\begin{proof}
For any map from $X^{2}$ into $X$ this is an immediate consequence of [3, theorem 3]. One can check the proof of this theorem to see that a small adjustment in definitions would give the same result for any map from $X\times Y$ into $Z$. To see the second part, note that the closures of this rectangles are still rectangles, and since $\rho$ is continuous, it depends on at most one coordinate on each of this closures. By [4, theorem 8.2] we can assume these rectangles are clopen.
\end{proof}
\begin{theorem}\label{10}
 If $\rho$ is a continuous map from $X\times Y$ into $Z$ where $X$ and $Y$ are compact topological spaces and $Z$ is a sub-Stonean space, then  $X\times Y$ can be covered by finitely many mutually disjoint clopen rectangles such that $\rho$ depends on at most one coordinate on each of them.
\end{theorem}
\begin{proof}
 We just need to show that the case $(2)$ of the lemma \ref{51} does not happen. Suppose $\{x_{i}\}$ and $\{y_{i}\}$ are sequences guaranteed by $(2)$. Define the map $g : \mathbb{N}^{2}\longrightarrow X\times Y$ by $g(m,n) = (x_{m} , y_{n})$. Then  $g$ continuously extends to a map $\beta g : \beta \mathbb{N}^{2}\longrightarrow X\times Y$ and the continuous map $h : \beta \mathbb{N}^{2}\longrightarrow Z$ defined by
$h = \rho \circ \beta g$ has the property that $h(l,l)\neq h(m,n)$  for all $l$ and all $m,n$ such that $m<n$. This contradicts Corollary $7.6$ in \cite{Fdim} which states that if $h: \beta\mathbb{N}^{2}\longrightarrow Z$ is a continuous map and $Z$ is a sub-Stonean space, then the sets $h([\mathbb{N}]^{2})$ and $h(\Delta^{2}\mathbb{N})$ have nonempty intersection.
\end{proof}
As a corollary of this, if $X$ and $Y$ are infinite, any such $\rho$ is not injective. For C*-algebras, the product of non-commutative spaces corresponds to the tensor product of algebras. By the Gelfand transform we can restate Farah's theorem in terms of commutative C*-algebras.
\nonumber \begin{theorem}\label{1}
 Suppose  $f: \mathcal{A}\rightarrow \mathcal{B}\otimes \mathcal{C}$ is a unital *-homomorphism, where $\mathcal{A}, \mathcal{B}$ and $\mathcal{C}$ are unital commutative C*-algebras and $\mathcal{A}$ is a $SAW^{*}$-algebra. Then there are finitely many projections $p_{1}, \dots p_{s}$ in $\mathcal{B}$ and $q_{1}, \dots q_{t}$ projections in $\mathcal{C}$  such that $\sum_{i=1}^{s} p_{i} = 1_{\mathcal{B}}$ and  $\sum_{i=1}^{t} q_{i} = 1_{\mathcal{C}}$ and for every $1\leq i\leq s$ and $1\leq j \leq t$ and either for every $a\in \mathcal{A}$ we have $(p_{i}\otimes q_{j})f(a)\in (p_{i}\mathcal{B}p_{i})\otimes q_{j}$ or for every $a\in \mathcal{A}$ we have $(p_{i}\otimes q_{j})f(a)\in p_{i} \otimes (q_{j}\mathcal{C}q_{j})$.
 \end{theorem}

 Note that in particular every element in the image of $f$ is a finite sum of elementary tensor products and if $\mathcal{A}$ is a commutative $SAW^{*}$-algebra with no projections (e.g. $A=C(X)$ where $X$ is a connected sub-Stonean space like $\beta \mathbb{R} \setminus \mathbb{R}$) the image of $f$ can be identified with a C*-subalgebra of $\mathcal{B}$ or $\mathcal{C}$.

\begin{lemma}\label{5}

If $\mathcal{B}$ is an infinite-dimensional, unital C*-algebra, we can find an orthogonal
sequence $\{a_{1}, a_{2}, \dots \}$ in $\mathcal{B}$ such that $0\leq a_{i}\leq 1_{\mathcal{B}}$ for all $i$ and a sequence of states on $\mathcal{B}$,
$\{ \phi_{n} \}$, such that $\phi_{n}(a_{n}) = 1$ and $\phi_{n}(a_{m}) = 0 $ if $m \neq n$.
\end{lemma}
\begin{proof}
 It is well-known that any maximal abelian subalgebra (MASA) of an infinite-dimensional C*-algebra $\mathcal{B}$ is also infinite-dimensional. If not then there are orthogonal 1-dimensional projections $\{ p_{1}, p_{2}, \ldots, p_{n}\}$ in the MASA such that $\sum_{i=1}^{n} p_{i}=1_{\mathcal{B}}$. Since $\mathcal{B}=$ $\sum_{i,j=1}^{n} {p_i}\mathcal{B}$${p_{j}}$ and for each pair $i,j$, we have $p_{i}\mathcal{B}p_{j}$ is either $\{0\}$ or 1-dimensional, $\mathcal{B}$ is finite-dimensional. Fix such a MASA, and by the Gelfand-Naimark theorem identify it with $C(X)$ for some  compact Hausdorff space $X$. We can also identify the set of pure states of $C(X)$ with $X$. Since $X$ is an infinite normal space we can choose a discrete sequence of pure states $\{\phi_{n}\}$ in $X$ and find a pairwise disjoint sequence $\{U_{n}\}$ of open neighborhoods of $\{\phi_{n}\}$. By  Uryshon's lemma we get an orthogonal sequence $0\leq a_{n}\leq 1_{\mathcal{B}}$ in $C(X)$ such that $\phi_{n}(a_{n})= a_{n}(\phi_{n})=1$  and $a_{n}$ vanishes outside of $U_{n}$. So $\phi_{n}(a_{m})= a_{m}(\phi_{n})=0$ if $m\neq n$. Now by the Hahn-Banach extension theorem extend $\phi_{n}$ to a functional on $\mathcal{B}$ of norm $1$. Since $\phi_{n}(1_{\mathcal{B}})=1$, this extension is a state.
\end{proof}

Note that if $\phi$ is a state on a C*-algebra $\mathcal{A}$ and $\phi(a)=1$ for $0\leq a \leq 1_{\mathcal{A}}$, as a consequence of the Cauchy-Schwartz inequality for states  we have $\phi(b)= \phi(aba) $ for any $b\in \mathcal{A}$ (cf. \cite{FW}, lemma 4.8).
\begin{lemma}\label{6}
“Let $\{\phi_{n}\}$ be a sequence of states on a $SAW^{*}$-algebra $\mathcal{A}$. If there exists a sequence
$\{a_{n}\}$ of mutually orthogonal positive elements in $\mathcal{A}$ such that $\| a_{n}\|=\phi_{n}(a_{n}) = 1$ and $\phi_{n}(a_{m}) = 0 $ if $m \neq n$ then the weak*-closure of ${\{\phi_{n}\}}$
is homeomorphic to $\beta \mathbb{N}$.
\end{lemma}

\begin{proof}
Let $D$ be a subset of $\mathbb{N}$. We show that $\overline{\{\phi_{n} : n \in D \}} \cap \overline{\{\phi_{n} : n \in D^{c} \}} =\emptyset$.
Take $\psi \in \overline{\{\phi_{n} : n \in D \}}$. Let $a=\sum_{i\in D} 2^{-i}a_{i}$ and $b=\sum_{i\in D^{c}} 2^{-i}a_{i}$. Since $\mathcal{A}$
is a $SAW^{*}$-algebra, there exists a positive $e\in\mathcal{A}$ such that  $ea= a$ and $eb=0$. Then   $ea_{n}= a_{n}$ for $n \in D$ and $ea_{n}=0$ for every $n \in  D^{c}$.
For $n\in D$ we have $\phi_{n}(e)= \phi_{n}({ea_{n}})= \phi_{n}({a_{n}})= 1$  and for $n \in D^{c}$
we have $\phi_{n}(e)= \phi_{n}({ea_{n}})= \phi_{n}(0)= 0$. Hence $\psi(e) = 1$ and $\psi$ is not in $\overline{\{\phi_{n} : n \in D^{c} \}}$.

 Now let $F : \beta \mathbb{N}\longrightarrow \overline{\{\phi_{n} : n \in \mathbb{N}\}}$ be the continuous map such that $F(n)=\phi_{n}$. Let $\mathcal{U}$ and $\mathcal{V}$ be two distinct ultrafilters in $\beta\mathbb{N}$ and pick $X\subseteq\mathbb{N}$ such that $X \in \mathcal{U}$ but $X$ is not in $\mathcal{V}$. Therefore
 \begin{equation}
 F({\mathcal{U}}) \in F(\overline{X}) \subseteq {\overline{F(X)}} = \overline{\{\phi_{n} : n \in X\}}. \nonumber
\end{equation}
Similarly $F({\mathcal{V}})\in  \overline{\{\phi_{n} : n \in X^{c}\}}$. So $F$ is injective and clearly surjective. Since $\beta\mathbb{N}$ is compact and  $\overline{\{\phi_{n}\}}$ is Hausdorff, it follows that $F$ is a homeomorphism.
\end{proof}
\begin{theorem}
Any $SAW^{*}$ algebra is essentially non-factorizable.
\end{theorem}
\begin{proof}
Let $\mathcal{A}$ be a $SAW^{*}$-algebra. Suppose that $\mathcal{A}= \mathcal{B}\otimes_{\nu}\mathcal{C}$  for the C*-completion of
the algebraic tensor product $\mathcal{B}\odot\mathcal{C}$ of infinite dimensional C*-algebras $\mathcal{B}$ and $\mathcal{C}$ with
respect to some C*-norm $\|.\|_{\nu}$ . By Lemma \ref{5} there are orthogonal sequences of positive
contractions $\{b_{n}\}\subseteq \mathcal{B}$ and $\{c_{n}\}\subseteq \mathcal{C}$ and sequences $\{\phi_{n}\}\subseteq \mathcal{B}^{*}$ and $\{\psi_{n}\}\subseteq \mathcal{C}^{*}$ such
that $\phi_{n}(b_{n})=\psi_{n}(c_{n})=1$ and $\phi_{n}(b_{m})=\psi_{n}(c_{m})=0$ for $m\neq n$. Identifying $\mathcal{A}$ with $\mathcal{B}\otimes_{\nu}\mathcal{C}$
and letting $a_{m,n}=b_{m}\otimes c_{n}$ and $\gamma_{m,n}=\phi_{m}\otimes\psi_{n}$, it is immediate that\\
\begin{equation}
\nonumber\gamma_{m,n}(a_{m^{\prime},n^{\prime}})=
\left\{\begin{matrix}
 & 1 & (m, n) = (m^{\prime},n^{\prime}) \\
 & 0 & (m, n) \neq (m^{\prime},n^{\prime})
\end{matrix}\right. .
\end{equation}\\
Let $X=\overline{\{\phi_{n}: n\in\mathbb{N}\}}^{w^{*}}$ and $Y=\overline{\{\psi_{n}: n\in\mathbb{N}\}}^{w^{*}}$.
Then $X$ and $Y$ are compact subsets
of $\mathcal{B}^{*}$ and $\mathcal{C}^{*}$, respectively, $\{\phi_{n} : n\in\mathbb{N}\} \times \{\psi_{n} : n\in\mathbb{N} \}$ is a dense subset of $\overline{\{\gamma_{m,n}\}}^{w^{*}}$
and by compactness $X\times Y$ is homeomorphic to $\overline{\{\gamma_{m,n}\}}^{w^{*}}$, which is homeomorphic to $\beta\mathbb{N}$
by Lemma \ref{6}. This contradicts the remark following the proof of theorem \ref{10}. Hence there is no *-isomorphism $\mathcal{A}\cong \mathcal{B}\otimes\mathcal{C}$ with $\mathcal{B}$ and $\mathcal{C}$ infinite dimensional.
\end{proof}
\begin{corollary}
The Calkin algebra is essentially non-factorizable.
\end{corollary}

We don't know if theorem \ref{1} is true for non-commutative C*-algebras. But an analogous theorem for non-commutative C*-algebras would provide us with a strong tool to study the automorphisms between tensorial powers of the Calkin algebra or other $SAW^{*}$-algebras such as ultrapowers of C*-algebras.

\end{document}